\documentclass[11pt]{amsart}
\usepackage{amssymb, amsbsy}
\usepackage{amsfonts, amscd}
\usepackage{mathtools}
\usepackage{geometry}
\numberwithin{equation}{section}

\usepackage[usenames,dvipsnames]{xcolor}
\usepackage{hyperref}
\hypersetup{
    colorlinks=true,
    citecolor=red,
    linkcolor=blue,
    filecolor=magenta,
    urlcolor=red,
}

\newcommand{\Fil}{\operatorname{Fil}}
\newcommand{\ord}{\operatorname{ord}}
\newcommand{\fm}{\mathfrak{m}}
\newcommand{\fa}{\mathfrak{a}}
\newcommand{\fb}{\mathfrak{b}}
\newcommand{\fc}{\mathfrak{c}}
\newcommand{\fp}{\mathfrak{p}}
\newcommand{\fq}{\mathfrak{q}}

\newtheorem{thm}{Theorem}[section]

\newtheorem{lem}[thm]{Lemma}

\theoremstyle{definition}

\newtheorem{ques}[thm]{Question}
\newtheorem{rem}[thm]{Remark}

\newtheorem{cons}[thm]{Construction}

\title{On two questions of Qi on saturated filtrations}

\author{Jihao Liu}
\address{Department of Mathematics, Peking University, No. 5 Yiheyuan Road, Haidian District, Beijing 100871, China}
\address{Beijing International Center for Mathematical Research, Peking University, No. 5 Yiheyuan Road, Haidian District, Beijing 100871, China}
\email{liujihao@math.pku.edu.cn}

\author{Haofeng Zhang}
\address{Department of Mathematics, Peking University, No. 5 Yiheyuan Road, Haidian District, Beijing 100871, China}
\address{Institute of Mathematics, Academy of Mathematics and Systems Science, Chinese Academy of Sciences, Beijing 100190, China}
\email{april@aprilg.moe}

\date{\today}

\begin{document}

\begin{abstract}
We answer two questions of Qi [J. Reine Angew. Math. 830 (2026)] on the metric space of saturated filtrations, one affirmatively and one negatively. The main result of this paper was obtained by Chatgpt 5.5 pro, and the Danus system based on the Rethlas system.
\end{abstract}

\maketitle

\section{Introduction}\label{sec:introduction}

Let $(R,\fm)$ be a Noetherian local ring of dimension $n$. An \emph{$\fm$-filtration} is a decreasing family $\fa_\bullet=(\fa_\lambda)_{\lambda \in \mathbb{R}^+}$ of $\fm$-primary ideals together with $\fa_0 := R$ satisfying $\fa_\lambda\fa_\mu\subseteq\fa_{\lambda+\mu}$ and $\fa_\lambda=\bigcap_{\mu<\lambda}\fa_\mu$. Its \emph{(normalized) multiplicity} is
\[
e(\fa_\bullet)=\lim_{m\to\infty}\frac{\ell(R/\fa_m)}{m^n/n!},
\]
which exists for linearly bounded filtrations by \cite[\S2.1.5]{Qi26} (see also \cite[Theorem~3.8]{LM09}). Following Qi \cite[Lemma~2.13 and \S2.2]{Qi26}, write $\Fil_{R,\fm}$ for the space of \emph{linearly bounded} $\fm$-filtrations, equivalently those of positive multiplicity, and $\Fil^s_{R,\fm}$ for the subspace of \emph{saturated} filtrations, that is, those that are intersections of valuation filtrations. On $\Fil_{R,\fm}$ Qi introduced the function
\begin{equation}\label{eq:d1}
d_1(\fa_\bullet,\fb_\bullet)=2e(\fa_\bullet\cap\fb_\bullet)-e(\fa_\bullet)-e(\fb_\bullet),
\end{equation}
a multiplicity-theoretic analogue of the Darvas $d_1$-metric in K\"ahler geometry~\cite{Dar15}; it is a pseudometric on $\Fil_{R,\fm}$ and restricts to a genuine metric on $\Fil^s_{R,\fm}$ \cite[Theorem~1.1]{Qi26}.

A decreasing sequence $\fa_{\bullet,1}\supseteq\fa_{\bullet,2}\supseteq\cdots$ in $\Fil^s_{R,\fm}$ always has a termwise intersection $\fa_\bullet=\bigcap_{k\geq 1}\fa_{\bullet,k}$, and when it is still an $\fm$-filtration, it is moreover saturated by \cite[Corollary~2.21]{Qi26}. It is natural to ask whether the metric $d_1$ is compatible with such monotone limits, and Qi raised the following.

\begin{ques}[{\cite[Question~6.5]{Qi26}}]\label{ques:monotone}
Let $\{\fa_{\bullet,k}\}_{k\geq 1}$ be a decreasing sequence in $\Fil^s_{R,\fm}$ such that $\fa_\bullet:=\bigcap_{k\geq 1}\fa_{\bullet,k}$ belongs to $\Fil^s_{R,\fm}$. Must $d_1(\fa_\bullet,\fa_{\bullet,k})\to 0$ as $k\to\infty$?
\end{ques}

We answer Question~\ref{ques:monotone} in the negative, already on the two-dimensional smooth local ring.

\begin{thm}\label{thm:monotone}
Let $R=\mathbb{C}[[x,y]]$ and $\fm=(x,y)$. Then there is a decreasing sequence $\{\fa_{\bullet,k}\}_{k\geq 1}$ in $\Fil^s_{R,\fm}$ whose termwise intersection $\fa_\bullet=\bigcap_{k\geq 1}\fa_{\bullet,k}$ lies in $\Fil^s_{R,\fm}$, yet
\[
d_1(\fa_\bullet,\fa_{\bullet,k})\longrightarrow\tfrac12\neq 0.
\]
In particular, monotone convergence fails for $(\Fil^s_{R,\fm},d_1)$.
\end{thm}

The sequence is explicit. Fix pairwise distinct $t_i\in\mathbb{C}$, set $\ell_i=y-t_ix$ and $M_i=2^i$, and let $v_i$ be the monomial valuation with weights $(1,1+M_i)$ in the regular parameters $(x,\ell_i)$. The finite intersections $\fa_{\bullet,k}=\bigcap_{i=1}^k\fa_\bullet(v_i)$ form a decreasing sequence of saturated filtrations with $e(\fa_{\bullet,k})=\tfrac{1-2^{-k}}{2-2^{-k}}\to\tfrac12$. Because the linear forms $\ell_i$ point in infinitely many directions, the termwise intersection collapses to the $\fm$-adic filtration $\fm_\bullet$, which is saturated of multiplicity $1$; since $\fm_\bullet\subseteq\fa_{\bullet,k}$, the containment formula \eqref{eq:containment} gives $d_1(\fm_\bullet,\fa_{\bullet,k})=1-e(\fa_{\bullet,k})\to\tfrac12$. Thus the intersection sits at $d_1$-distance bounded away from every term of the sequence.

Qi also proved that, when $R$ is analytically irreducible, $(\Fil^s_{R,\fm},d_1)$ is a geodesic metric space, and asked the following.

\begin{ques}[{\cite[Question~6.6]{Qi26}}]\label{ques:contractible}
Is the metric space $(\Fil^s_{R,\fm},d_1)$ contractible?
\end{ques}

We answer Question~\ref{ques:contractible} affirmatively, in full generality.

\begin{thm}\label{thm:contractible}
Let $(R,\fm)$ be a Noetherian analytically irreducible local domain. Then the metric space $(\Fil^s_{R,\fm},d_1)$ is contractible.
\end{thm}

The proof is an explicit contraction (null-homotopy of the identity) to the saturated $\fm$-adic filtration $o_\bullet=\widetilde{\fm^\bullet}$. Writing $T_c$ for the rescaling $T_c(\fa_\bullet)_\lambda=\fa_{c\lambda}$, we contract along
\[
H(\fa_\bullet,t)=T_t\fa_\bullet\cap T_{1-t}o_\bullet,\qquad 0<t<1.
\]
The scaling identity $e(T_c\fa_\bullet)=c^n e(\fa_\bullet)$ and the rooftop inequality for $d_1$ make $H$ Lipschitz in the space variable and continuous in time, with $H(\,\cdot\,,0)\equiv o_\bullet$ and $H(\,\cdot\,,1)=\mathrm{id}$. In particular the argument uses neither a field assumption nor the geodesic structure of $\Fil^s_{R,\fm}$.

\begin{rem}
The sketch of the proof of the main result of this paper was obtained by Chatgpt 5.5 pro, and later summed up, verified, and properly written by the Danus system, a specialized agent built on Rethlas and substantially more capable for fundamental mathematical research based on the Rethlas system. Human verification and polishing were done afterwards. See \cite{Ju26} for a detailed introduction to the Rethlas system. Due to the limitation of automated systems, it is possible that we have missed some related references in the literature, and we welcome any comments from experts.
\end{rem}

\subsection*{Acknowledgements}
The first author was partially supported by the National Key R\&D Program of China \#\allowbreak 2024YFA1014400. The first author would like to thank the Rethlas team, namely Haocheng Ju, Jiedong Jiang, Shurui Liu, Guoxiong Gao, Yuefeng Wang, Zeming Sun, Bin Wu, Liang Xiao, and Bin Dong, for their contributions to the development of Rethlas and its customized version used for the problem studied in this paper. The authors would like to thank Yuan Lu for useful discussions. The first author would like to thank Ruochuan Liu and Gang Tian for constant support and encouragement. The second author would like to thank Jie Liu for constant support.

\section{Preliminaries}\label{sec:preliminaries}

In this section, we collect the facts from \cite{Qi26} that will be used in the proofs of Theorems~\ref{thm:monotone} and~\ref{thm:contractible}. We will use freely that $d_1$ is a pseudometric on $\Fil_{R,\fm}$ that restricts to a metric on $\Fil^s_{R,\fm}$. For $\fa_\bullet\subseteq\fb_\bullet$ we have
\begin{equation}\label{eq:containment}
d_1(\fa_\bullet,\fb_\bullet)=e(\fa_\bullet)-e(\fb_\bullet).
\end{equation}
We also use that a filtration lies in $\Fil^s_{R,\fm}$ if and only if it is an intersection of valuation filtrations $\fa_\bullet(v)$ with $v$ centered at $\fm$ of positive volume \cite[Proposition~2.20]{Qi26}, and that an intersection of saturated filtrations is saturated whenever it remains linearly bounded \cite[Corollary~2.21]{Qi26}. Finally, we use the rooftop estimate
\begin{equation}\label{eq:rooftop}
d_1(\fc_\bullet\cap\fa_\bullet,\fc_\bullet\cap\fb_\bullet)\leq d_1(\fa_\bullet,\fb_\bullet)
\end{equation}
from \cite[Lemma~3.3]{Qi26}.

\section{Failure of monotone convergence}\label{sec:monotone}

In this section, we prove Theorem~\ref{thm:monotone}. Throughout this section $R=\mathbb{C}[[x,y]]$, $\fm=(x,y)$, and $n=\dim R=2$. We use the terminology recalled in Sections~\ref{sec:introduction} and~\ref{sec:preliminaries}.

\begin{cons}\label{cons:sequence}
Fix pairwise distinct $t_i\in\mathbb{C}$ for $i\geq 1$, and set
\[
\ell_i=y-t_ix,\qquad M_i=2^i.
\]
Since $(x,\ell_i)$ is a regular system of parameters of $R$, every $0\neq f\in R$ has a unique expansion $f=\sum_{a,b\geq 0}c_{a,b}\,x^a\ell_i^{\,b}$ with $c_{a,b}\in\mathbb{C}$. Define the monomial valuation
\[
v_i(f)=\min\{a+(1+M_i)b\mid c_{a,b}\neq 0\},\qquad v_i(0)=+\infty,
\]
of weights $(1,1+M_i)$ in the parameters $(x,\ell_i)$, with valuation filtration $\fa_\lambda(v_i)=\{f\in R\mid v_i(f)\geq\lambda\}$. For each integer $k\geq 1$ put
\[
\fa_{\bullet,k}=\bigcap_{i=1}^k\fa_\bullet(v_i).
\]
\end{cons}

\begin{lem}\label{lem:counterexample-saturated}
Each $v_i$ is centered at $\fm$ with $e(\fa_\bullet(v_i))=\tfrac{1}{M_i+1}>0$. Consequently every $\fa_{\bullet,k}$ is a linearly bounded saturated filtration, and $\fa_{\bullet,k+1}\subseteq\fa_{\bullet,k}$.
\end{lem}

\begin{proof}
For $0\neq g\in R$ one has $v_i(g)\geq\ord_\fm(g)$, with $v_i(g)>0$ exactly when $g\in\fm$; hence $v_i$ is centered at $\fm$. The single-valuation multiplicity is computed as in Lemma~\ref{lem:counterexample-multiplicity} below with $k=1$ and weight $M_i$, giving $e(\fa_\bullet(v_i))=\tfrac{(1/M_i)}{1+(1/M_i)}=\tfrac{1}{M_i+1}>0$, so $v_i$ is a valuation of positive volume. By \cite[Proposition~2.20]{Qi26} the finite intersection $\fa_{\bullet,k}=\bigcap_{i=1}^k\fa_\bullet(v_i)$ is therefore saturated, provided it is linearly bounded. The latter holds: for $i=1$ the weight is $1+M_1=3$, so $v_1(f)\leq 3\ord_\fm(f)$ for all $f$, whence $f\in(\fa_{\bullet,k})_\lambda$ forces $v_1(f)\geq\lambda$ and thus $f\in\fm^{\lceil\lambda/3\rceil}$. Finally $\fa_{\bullet,k+1}=\fa_{\bullet,k}\cap\fa_\bullet(v_{k+1})\subseteq\fa_{\bullet,k}$.
\end{proof}

\begin{lem}\label{lem:counterexample-multiplicity}
For every integer $k\geq 1$,
\[
e(\fa_{\bullet,k})=\frac{S_k}{1+S_k}=\frac{1-2^{-k}}{2-2^{-k}},\qquad S_k:=\sum_{i=1}^k\frac{1}{M_i}=1-2^{-k}.
\]
\end{lem}

\begin{proof}
Fix $k$ and an integer $q\geq 1$. For a homogeneous form $P$ of degree $d$, write $P=\sum_{b=0}^d c_b\,x^{d-b}\ell_i^{\,b}$ in the parameters $(x,\ell_i)$; since $v_i(x^{d-b}\ell_i^{\,b})=d+M_ib$, we have
\[
v_i(P)\geq q\iff \ell_i^{\,r_{i,d,q}}\mid P,\qquad r_{i,d,q}:=\max\left(0,\Bigl\lceil\frac{q-d}{M_i}\Bigr\rceil\right).
\]
As $\ell_1,\dots,\ell_k$ are pairwise nonproportional linear forms, $v_i(P)\geq q$ for all $1\leq i\leq k$ if and only if $\prod_{i=1}^k\ell_i^{\,r_{i,d,q}}\mid P$, so the subspace of degree-$d$ forms satisfying all $k$ conditions has codimension $\min\bigl(d+1,\sum_{i=1}^k r_{i,d,q}\bigr)$ in the $(d+1)$-dimensional space of degree-$d$ forms.

Write $f=\sum_{d\geq 0}P_d$ for the decomposition into homogeneous parts. Since $v_i(f)=\min_d v_i(P_d)$, the condition $v_i(f)\geq q$ for all $i\leq k$ is imposed degree by degree. Moreover $\fm^q\subseteq(\fa_{\bullet,k})_q$, so only the degrees $0\leq d<q$ contribute to $R/(\fa_{\bullet,k})_q$, and
\[
\ell\bigl(R/(\fa_{\bullet,k})_q\bigr)=\sum_{d=0}^{q-1}\min\Bigl(d+1,\ \sum_{i=1}^k\Bigl\lceil\frac{q-d}{M_i}\Bigr\rceil\Bigr).
\]
The ceiling sum differs from $S_k(q-d)$ by at most $k$ uniformly in $d$, so the total error is $O(kq)$ and is negligible after division by $q^2$. Writing $t=d/q$, the Riemann sum converges:
\[
\lim_{q\to\infty}\frac{\ell(R/(\fa_{\bullet,k})_q)}{q^2}=\int_0^1\min\{t,\,S_k(1-t)\}\,dt=\frac{S_k}{2(1+S_k)},
\]
where the two integrands meet at $t=S_k/(1+S_k)$. Since $n=2$, the normalization gives
\[
e(\fa_{\bullet,k})=\lim_{q\to\infty}\frac{\ell(R/(\fa_{\bullet,k})_q)}{q^2/2}=\frac{S_k}{1+S_k}=\frac{1-2^{-k}}{2-2^{-k}}.\qedhere
\]
\end{proof}

\begin{lem}\label{lem:counterexample-intersection}
The termwise intersection is the $\fm$-adic filtration:
\[
\bigcap_{k\geq 1}\fa_{\bullet,k}=\fm_\bullet,\qquad (\fm_\bullet)_\lambda=\fm^{\lceil\lambda\rceil}.
\]
Moreover $\fm_\bullet\in\Fil^s_{R,\fm}$ and $e(\fm_\bullet)=1$.
\end{lem}

\begin{proof}
Let $\fb_\lambda=\bigcap_{k\geq 1}(\fa_{\bullet,k})_\lambda=\bigcap_{i\geq 1}\{f\mid v_i(f)\geq\lambda\}$. Since $v_i\geq\ord_\fm$, we have $\fm^{\lceil\lambda\rceil}\subseteq\fb_\lambda$. Conversely, let $0\neq f\notin\fm^{\lceil\lambda\rceil}$ and set $d=\ord_\fm(f)<\lambda$, with lowest homogeneous part $P_d$. As $P_d$ has only finitely many linear factors while the $\ell_i$ point in infinitely many distinct directions, we may choose $i$ with $\ell_i\nmid P_d$; then the coefficient of $\ell_i^{\,0}$ in the $(x,\ell_i)$-expansion of $P_d$ is nonzero, so $v_i(P_d)=d$, while every higher homogeneous part has $v_i$-value at least $d+1$. Hence $v_i(f)=d<\lambda$ and $f\notin\fb_\lambda$. This proves $\fb_\lambda=\fm^{\lceil\lambda\rceil}$.

The $\fm$-adic filtration $\fm_\bullet$ is the valuation filtration of $\ord_\fm$: for nonzero power series the lowest homogeneous part of a product is the product of the lowest homogeneous parts, so $\ord_\fm(fg)=\ord_\fm(f)+\ord_\fm(g)$ and $\fm^{\lceil\lambda\rceil}=\{f\mid\ord_\fm(f)\geq\lambda\}$. Thus $\ord_\fm$ is a valuation centered at $\fm$, and from $\ell(R/\fm^q)=q(q+1)/2$ we get $e(\fm_\bullet)=\lim_{q\to\infty}\ell(R/\fm^q)/(q^2/2)=1>0$. By \cite[Proposition~2.20]{Qi26}, $\fm_\bullet\in\Fil^s_{R,\fm}$.
\end{proof}

\begin{proof}[Proof of Theorem~\ref{thm:monotone}]
By Lemma~\ref{lem:counterexample-saturated} the sequence $\{\fa_{\bullet,k}\}_{k\geq 1}$ of Construction~\ref{cons:sequence} is a decreasing sequence in $\Fil^s_{R,\fm}$, and by Lemma~\ref{lem:counterexample-intersection} its termwise intersection is $\fm_\bullet\in\Fil^s_{R,\fm}$. Since $\fm_\bullet=\bigcap_{i\geq 1}\fa_\bullet(v_i)\subseteq\fa_{\bullet,k}$, the two filtrations are comparable, so the containment formula \eqref{eq:containment} and Lemmas~\ref{lem:counterexample-multiplicity} and~\ref{lem:counterexample-intersection} give
\[
d_1(\fm_\bullet,\fa_{\bullet,k})=e(\fm_\bullet)-e(\fa_{\bullet,k})=1-\frac{1-2^{-k}}{2-2^{-k}}\xrightarrow[\ k\to\infty\ ]{}1-\frac12=\frac12\neq 0.
\]
Hence the decreasing sequence $\{\fa_{\bullet,k}\}_{k\geq 1}$ in $\Fil^s_{R,\fm}$, whose termwise intersection $\fm_\bullet$ again lies in $\Fil^s_{R,\fm}$, does not converge to that intersection in $d_1$. This answers Question~\ref{ques:monotone} in the negative.
\end{proof}

\section{Contractibility}\label{sec:contractibility}

In this section, we prove Theorem~\ref{thm:contractible}. Throughout this section $(R,\fm)$ is a Noetherian analytically irreducible local domain of dimension $n$, and we use the terminology and facts recalled in Sections~\ref{sec:introduction} and~\ref{sec:preliminaries}. Let
\[
o_\bullet:=\widetilde{\fm^\bullet}
\]
be the saturation of the $\fm$-adic filtration $\fm_\lambda=\fm^{\lceil\lambda\rceil}$, that is, the intersection of all valuation filtrations centered at $\fm$ that contain $\fm_\bullet$. Being such an intersection, $o_\bullet$ is saturated, and it is the canonical element of $\Fil^s_{R,\fm}$ \cite[\S2.2]{Qi26}; in particular $o_\bullet\in\Fil^s_{R,\fm}$. As an element of $\Fil^s_{R,\fm}\subseteq\Fil_{R,\fm}$, the filtration $o_\bullet$ is linearly bounded, hence of positive multiplicity, so $0<e(o_\bullet)$. Throughout this section we use that $R$ is analytically irreducible only through \cite[Theorem~1.1]{Qi26}, which guarantees that $d_1$ restricts to a genuine metric on $\Fil^s_{R,\fm}$.

\begin{lem}[Scaling]\label{lem:scaling}
For $c>0$, define $T_c(\fa_\bullet)_\lambda=\fa_{c\lambda}$. Then $T_c$ preserves $\Fil_{R,\fm}$ and $\Fil^s_{R,\fm}$, and
\[
T_c(\fa_\bullet\cap\fb_\bullet)=T_c\fa_\bullet\cap T_c\fb_\bullet,\qquad
e(T_c\fa_\bullet)=c^n e(\fa_\bullet),\qquad
d_1(T_c\fa_\bullet,T_c\fb_\bullet)=c^n d_1(\fa_\bullet,\fb_\bullet).
\]
\end{lem}

\begin{proof}
By definition $T_c\fa_\bullet$ is again an $\fm$-filtration. If $\fa_\bullet\subseteq\fm^{\alpha\bullet}$ for some $\alpha>0$, then $T_c\fa_\bullet\subseteq\fm^{\alpha c\bullet}$, so $T_c\fa_\bullet$ is linearly bounded.

Assume that $\fa_\bullet$ is saturated. By the characterization of saturated filtrations as intersections of valuation filtrations, there is a nonempty set $\Sigma$ of valuations centered at $\fm$ with $\fa_\bullet=\bigcap_{v\in\Sigma}\fa_\bullet(v)$. For each $v\in\Sigma$,
\[
T_c\bigl(\fa_\bullet(v)\bigr)_\lambda
=\{f\mid v(f)\geq c\lambda\}
=\{f\mid (v/c)(f)\geq\lambda\}
=\fa_\bullet(v/c)_\lambda,
\]
and $v/c$ is again centered at $\fm$. Hence $T_c\fa_\bullet=\bigcap_{v\in\Sigma}\fa_\bullet(v/c)$ is again an intersection of valuation filtrations, so $T_c\fa_\bullet$ is saturated.

The intersection identity holds termwise. For the multiplicity, monotonicity of $\ell(R/\,\cdot\,)$ gives
\[
\ell(R/\fa_{\lfloor cm\rfloor})\leq\ell(R/\fa_{cm})\leq\ell(R/\fa_{\lceil cm\rceil}),
\]
and after division by $m^n/n!$ the two outer terms converge to $c^n e(\fa_\bullet)$, since $\lfloor cm\rfloor/m,\lceil cm\rceil/m\to c$. Since $\ell(R/\fa_{cm})$ is sandwiched between the two outer terms, it has the same limit; hence $e(T_c\fa_\bullet)=c^n e(\fa_\bullet)$. The formula for $d_1$ follows from \eqref{eq:d1}, the intersection identity, and the multiplicity-scaling formula.
\end{proof}

\begin{lem}[Multiplicity estimates]\label{lem:multiplicity-estimates}
For all $\fp_\bullet,\fq_\bullet\in\Fil_{R,\fm}$,
\[
e(\fp_\bullet\cap\fq_\bullet)\leq e(\fp_\bullet)+e(\fq_\bullet),
\qquad\text{and}\qquad
|e(\fp_\bullet)-e(\fq_\bullet)|\leq d_1(\fp_\bullet,\fq_\bullet).
\]
\end{lem}

\begin{proof}
For every $\lambda>0$ the natural map
\[
R/(\fp_\lambda\cap\fq_\lambda)\longrightarrow R/\fp_\lambda\oplus R/\fq_\lambda
\]
is injective, so $\ell\bigl(R/(\fp_\lambda\cap\fq_\lambda)\bigr)\leq\ell(R/\fp_\lambda)+\ell(R/\fq_\lambda)$. Passing to the asymptotic multiplicity gives $e(\fp_\bullet\cap\fq_\bullet)\leq e(\fp_\bullet)+e(\fq_\bullet)$.

Applying the containment formula \eqref{eq:containment} to $\fp_\bullet\cap\fq_\bullet\subseteq\fp_\bullet$ and $\fp_\bullet\cap\fq_\bullet\subseteq\fq_\bullet$, the quantities
\[
e(\fp_\bullet\cap\fq_\bullet)-e(\fp_\bullet),\qquad e(\fp_\bullet\cap\fq_\bullet)-e(\fq_\bullet)
\]
are nonnegative and, by \eqref{eq:d1}, sum to $d_1(\fp_\bullet,\fq_\bullet)$. Their difference is $e(\fq_\bullet)-e(\fp_\bullet)$, whence $|e(\fp_\bullet)-e(\fq_\bullet)|\leq d_1(\fp_\bullet,\fq_\bullet)$.
\end{proof}

\begin{lem}[Contraction estimates]\label{lem:contraction}
For $0<t<1$, set $H(\fa_\bullet,t)=T_t\fa_\bullet\cap T_{1-t}o_\bullet$. Then $H(\fa_\bullet,t)\in\Fil^s_{R,\fm}$, and:
\begin{enumerate}
\item $d_1\bigl(H(\fa_\bullet,t),H(\fb_\bullet,t)\bigr)\leq t^n\,d_1(\fa_\bullet,\fb_\bullet)$;
\item for $0<s,t<1$,
\[
d_1\bigl(H(\fa_\bullet,s),H(\fa_\bullet,t)\bigr)\leq |s^n-t^n|\,e(\fa_\bullet)+|(1-s)^n-(1-t)^n|\,e(o_\bullet);
\]
\item $d_1\bigl(H(\fa_\bullet,t),o_\bullet\bigr)\leq t^n e(\fa_\bullet)+\{1-(1-t)^n\}\,e(o_\bullet)$;
\item $d_1\bigl(H(\fa_\bullet,t),\fa_\bullet\bigr)\leq (1-t)^n e(o_\bullet)+\{1-t^n\}\,e(\fa_\bullet)$.
\end{enumerate}
\end{lem}

\begin{proof}
The factors $T_t\fa_\bullet$ and $T_{1-t}o_\bullet$ are saturated and linearly bounded by Lemma~\ref{lem:scaling}. Their termwise intersection is an $\fm$-filtration, is linearly bounded because it is contained in either factor, and is saturated by \cite[Corollary~2.21]{Qi26}; hence $H(\fa_\bullet,t)\in\Fil^s_{R,\fm}$.

(1) By the rooftop estimate \eqref{eq:rooftop} with fixed factor $T_{1-t}o_\bullet$ and then Lemma~\ref{lem:scaling},
\[
d_1\bigl(H(\fa_\bullet,t),H(\fb_\bullet,t)\bigr)\leq d_1(T_t\fa_\bullet,T_t\fb_\bullet)=t^n d_1(\fa_\bullet,\fb_\bullet).
\]

(2) Inserting the intermediate filtration $T_t\fa_\bullet\cap T_{1-s}o_\bullet$ and applying the triangle inequality and \eqref{eq:rooftop} to the two factors bounds the distance by
\[
d_1(T_s\fa_\bullet,T_t\fa_\bullet)+d_1(T_{1-s}o_\bullet,T_{1-t}o_\bullet).
\]
If $s\leq t$ then $T_t\fa_\bullet\subseteq T_s\fa_\bullet$, so by \eqref{eq:containment} and Lemma~\ref{lem:scaling} the first term equals $(t^n-s^n)e(\fa_\bullet)$; the opposite order gives the absolute value, and the same argument for $o_\bullet$ gives the stated bound.

(3) By the triangle inequality,
\[
d_1\bigl(H(\fa_\bullet,t),o_\bullet\bigr)\leq d_1\bigl(H(\fa_\bullet,t),T_{1-t}o_\bullet\bigr)+d_1\bigl(T_{1-t}o_\bullet,o_\bullet\bigr).
\]
Since $H(\fa_\bullet,t)=T_t\fa_\bullet\cap T_{1-t}o_\bullet\subseteq T_{1-t}o_\bullet$, the first term equals $e(H(\fa_\bullet,t))-e(T_{1-t}o_\bullet)$ by \eqref{eq:containment}, which is at most $e(T_t\fa_\bullet)=t^n e(\fa_\bullet)$ by Lemma~\ref{lem:multiplicity-estimates}. As $o_\bullet\subseteq T_{1-t}o_\bullet$, the second term equals $\{1-(1-t)^n\}e(o_\bullet)$ by \eqref{eq:containment} and Lemma~\ref{lem:scaling}.

(4) By the triangle inequality,
\[
d_1\bigl(H(\fa_\bullet,t),\fa_\bullet\bigr)\leq d_1\bigl(H(\fa_\bullet,t),T_t\fa_\bullet\bigr)+d_1\bigl(T_t\fa_\bullet,\fa_\bullet\bigr).
\]
As above the first term is at most $e(T_{1-t}o_\bullet)=(1-t)^n e(o_\bullet)$ by Lemma~\ref{lem:multiplicity-estimates}. Since $\fa_\bullet\subseteq T_t\fa_\bullet$, the second term equals $e(\fa_\bullet)-e(T_t\fa_\bullet)=\{1-t^n\}e(\fa_\bullet)$ by \eqref{eq:containment} and Lemma~\ref{lem:scaling}.
\end{proof}

\begin{proof}[Proof of Theorem~\ref{thm:contractible}]
Define $H\colon\Fil^s_{R,\fm}\times[0,1]\to\Fil^s_{R,\fm}$ by
\[
H(\fa_\bullet,t)=
\begin{cases}
o_\bullet, & t=0,\\[2pt]
T_t\fa_\bullet\cap T_{1-t}o_\bullet, & 0<t<1,\\[2pt]
\fa_\bullet, & t=1.
\end{cases}
\]
By Lemma~\ref{lem:contraction} the values lie in $\Fil^s_{R,\fm}$ for $0<t<1$, and the endpoint values do as well. We prove that $H$ is jointly continuous. Fix $(\fa_{\bullet,0},t_0)$. By Lemma~\ref{lem:multiplicity-estimates}, $|e(\fa_\bullet)-e(\fa_{\bullet,0})|\leq d_1(\fa_\bullet,\fa_{\bullet,0})$, so $e(\fa_\bullet)$ is locally bounded in the $d_1$-metric.

Suppose first $0<t_0<1$. For $t$ near $t_0$, the triangle inequality together with parts (1) and (2) of Lemma~\ref{lem:contraction} gives
\[
d_1\bigl(H(\fa_\bullet,t),H(\fa_{\bullet,0},t_0)\bigr)
\leq t^n d_1(\fa_\bullet,\fa_{\bullet,0})+|t^n-t_0^n|\,e(\fa_{\bullet,0})+|(1-t)^n-(1-t_0)^n|\,e(o_\bullet),
\]
whose right-hand side tends to $0$ as $(\fa_\bullet,t)\to(\fa_{\bullet,0},t_0)$.

At $t_0=0$, part (3) gives, for $0<t<1$,
\[
d_1\bigl(H(\fa_\bullet,t),o_\bullet\bigr)\leq t^n e(\fa_\bullet)+\{1-(1-t)^n\}e(o_\bullet),
\]
which tends to $0$ as $(\fa_\bullet,t)\to(\fa_{\bullet,0},0)$ by local boundedness of $e(\fa_\bullet)$; at $t=0$ the distance is $0$.

At $t_0=1$, part (4) and the triangle inequality give, for $0<t<1$,
\[
d_1\bigl(H(\fa_\bullet,t),\fa_{\bullet,0}\bigr)\leq (1-t)^n e(o_\bullet)+\{1-t^n\}e(\fa_\bullet)+d_1(\fa_\bullet,\fa_{\bullet,0}),
\]
which tends to $0$; at $t=1$ the distance is $d_1(\fa_\bullet,\fa_{\bullet,0})$, which also tends to $0$. Hence $H$ is jointly continuous.

Thus $H$ is a homotopy from the constant map at $o_\bullet$ to the identity of $\Fil^s_{R,\fm}$, so $(\Fil^s_{R,\fm},d_1)$ is contractible.
\end{proof}

\end{document}